\UseRawInputEncoding
\documentclass[12pt]{article}

\usepackage{amsfonts,amsmath,amsthm}
\usepackage{amssymb}
\usepackage{xcolor}
\usepackage{graphicx}
\usepackage{stmaryrd}
\usepackage{dsfont}

\usepackage[paper=a4paper,dvips,top=2cm,left=2cm,right=2cm,
foot=1cm,bottom=4cm]{geometry}

\newcommand{\vx}{{\bf x}}
\newcommand{\vy}{{\bf y}}

\newtheorem{Thm}{Theorem}[section]
\newtheorem{Def}[Thm]{Definition}
\newtheorem{Lem}[Thm]{Lemma}
\newtheorem{Prop}[Thm]{Proposition}

\newtheorem{remark}[Thm]{Remark}

\begin{document}

\title{Narrowing the Gap: SOS Ranks of \(4 \times 3\) Biquadratic Forms and a Lower Bound of $8$}

\author{Yi Xu\footnote{School of Mathematics, Southeast University, Nanjing  211189, China. Nanjing Center for Applied Mathematics, Nanjing 211135,  China. Jiangsu Provincial Scientific Research Center of Applied Mathematics, Nanjing 211189, China. ({\tt yi.xu1983@hotmail.com})}
		\and
		Chunfeng Cui\footnote{School of Mathematical Sciences, Beihang University, Beijing  100191, China.
			({\tt chunfengcui@buaa.edu.cn})}
		\and {and \
			Liqun Qi\footnote{
				Department of Applied Mathematics, The Hong Kong Polytechnic University, Hung Hom, Kowloon, Hong Kong.
				({\tt maqilq@polyu.edu.hk})}
		}
	}

\date{\today}
\maketitle

\begin{abstract}
We investigate the maximum sum-of-squares (SOS) rank of biquadratic forms in the critical case of \(4 \times 3\) variables, where the general bounds are currently \(7 \leq \mathrm{BSR}(4,3) \leq 11\). By analyzing two important structured subclasses, we obtain exact determinations and improved upper bounds that significantly narrow this gap.

For simple biquadratic forms  those containing only distinct terms of the type \(x_i^2 y_j^2\)  we prove that the maximum achievable SOS rank is exactly 7, a value attained by a form corresponding to a \(C_4\)-free bipartite graph with the maximum number of edges. This settles the question for simple forms.

For \(y\)-deficient biquadratic forms  a class introduced here that permits cross terms among two of the three \(y\)-variables while the third appears only in pure square terms  we prove an upper bound of 9 by combining Calder\"{o}n's theorem on \(m\times 2\) forms with the known value \(\mathrm{BSR}(4,2) = 5\).

Our main result is a constructive proof that \(\mathrm{BSR}(4,3) \geq 8\). We present an explicit non-simple, non-deficient \(4\times 3\) biquadratic form and prove it requires exactly eight squares, thereby improving the general lower bound. This shows that any form achieving a rank higher than 8 must possess a more complex algebraic structure, and it reduces the search space for determining the true value of \(\mathrm{BSR}(4,3)\). Connections to Zarankiewicz numbers, extremal graph theory, and classical results on sums of squares are highlighted throughout.

\medskip


\end{abstract}

\section{Introduction}
Let $m, n \ge 2$.
A biquadratic form in variables $\vx = (x_1,\dots,x_m)$ and $\vy = (y_1,\dots,y_n)$ is a homogeneous polynomial
\[
P(\vx,\vy) = \sum_{i,k=1}^{m} \sum_{j,l=1}^{n} a_{ijkl} x_i x_k y_j y_l,
\]
with real coefficients $a_{ijkl}$.
It is called \emph{positive semidefinite (PSD)} if $P(\vx,\vy) \ge 0$ for all $\vx,\vy$, and \emph{sum-of-squares (SOS)} if it can be written as a finite sum of squares of bilinear forms.
The smallest number of squares required is the \emph{SOS rank} of $P$, denoted $sos(P)$.   Let $BSR(m, n)$ be the maximum SOS rank of $m \times n$ SOS biquadratic forms     Currently, we know the following about $BSR(m, n)$.

\begin{itemize}
    \item $BSR(m, 2) = m+1$, i.e., $BSR(2, n) = n+1$ \cite{BPSV19}.
    \item $BSR(3, 3) = 6$ \cite{BSSV22}
    \item $BSR(m, n) \le mn-1$ \cite{QCX26}.
    \item $BSR(m, n) \ge z(m, n)$, where $z(m, n)$ is the Zarankiewicz number \cite{CQX26}.
\end{itemize}

Thus, the current frontier is on \(\mathrm{BSR}(4,3)\). As \(z(4,3) = 7\), we now know that \(7 \leq \mathrm{BSR}(4,3) \leq 11\). In this paper, we make substantial progress in narrowing this gap.

First, we conduct a systematic analysis of two fundamental subclasses. For \textbf{simple biquadratic forms} (Section 2), we leverage their connection to bipartite graphs and Zarankiewicz numbers to prove that the maximum SOS rank is exactly 7. For \textbf{\(y\)-deficient biquadratic forms} (Section 3), a new class introduced here, we use a splitting argument to establish an upper bound of 9.

Our main contribution is a constructive improvement of the general lower bound. In Section 4, we present an explicit \(4\times 3\) biquadratic form that is neither simple nor \(y\)-deficient and prove that its SOS rank is exactly 8 (Theorem 4.1). This result, \(\mathrm{BSR}(4,3) \geq 8\), is the central finding of our work. It demonstrates that the maximum rank for the general case is at least 8 and implies that any form achieving a rank of 9 or higher must possess a more complex algebraic structure than those considered here.

The remainder of this paper is organized as follows. In Section 2, we focus on simple biquadratic forms and prove that the maximum SOS rank for \(4\times 3\) simple forms is exactly 7. In Section 3, we introduce the class of \(y\)-deficient biquadratic forms, which may contain cross terms, and show that their SOS rank is at most 9. Diagonal forms appear as a special case, and we also mention an alternative proof using a row split. In Section 4, we construct and analyze the form \(Q\), proving it requires eight squares and establishing the new lower bound \(\mathrm{BSR}(4,3) \ge 8\). Finally, in Section 5, we conclude with a summary of our findings and a discussion of open problems and future research directions.

\section{Simple Biquadratic Forms}
Let $m \ge n$.
A biquadratic form is called \textbf{simple} if it contains only distinct terms of the type $x_i^2 y_j^2$.
We define a family of simple forms $P_{m,n,s}$ where $s = 1,\dots,mn$ counts the number of square terms, following a fixed ordering of index pairs $(i,j)$ (see \cite{QCX26} for details).
For $m=4$ and $n=3$ the first seven forms are:
\[
\begin{aligned}
P_{4,3,1} &= x_1^2 y_1^2, \\
P_{4,3,2} &= x_1^2 y_1^2 + x_2^2 y_2^2, \\
P_{4,3,3} &= x_1^2 y_1^2 + x_2^2 y_2^2 + x_3^2 y_3^2, \\
P_{4,3,4} &= x_1^2 y_1^2 + x_2^2 y_2^2 + x_3^2 y_3^2 + x_1^2 y_2^2, \\
P_{4,3,5} &= x_1^2 y_1^2 + x_2^2 y_2^2 + x_3^2 y_3^2 + x_1^2 y_2^2 + x_2^2 y_3^2, \\
P_{4,3,6} &= x_1^2 y_1^2 + x_2^2 y_2^2 + x_3^2 y_3^2 + x_1^2 y_2^2 + x_2^2 y_3^2 + x_3^2 y_1^2, \\
P_{4,3,7} &= x_1^2 y_1^2 + x_2^2 y_2^2 + x_3^2 y_3^2 + x_1^2 y_2^2 + x_2^2 y_3^2 + x_3^2 y_1^2 + x_4^2 y_1^2.
\end{aligned}
\]
Each of them corresponds to a $4 \times 3$ bipartite graph $G = (S, T, E)$, where the vertex sets $S = \{ 1, 2, 3, 4 \}$ and $T = \{ 1, 2, 3 \}$, and $E$ is the edge set.   For $P_{4, 3, 7}$, its edge set $E = \{ (1, 1), (2, 2), (3, 3), (1, 2), (2, 3), (3, 1), (4, 1) \}$.    All of these seven bipartite graphs are $C_4$-free, i.e., none of them contains a $C_4$ cycle $C_{ijkl} \equiv \{ (i, j), (i, l), (k, j), (k, l) \}$ with $i \neq k$ and $j \neq l$.   By \cite{CQX26}, the SOS rank of $P_{4,3,p}$ equals the size of its edge set.  In particular, as the bipartite graph corresponding to $P_{4,3,7}$ is a $4 \times 3$ bipartite graph with the maximum $C_4$-free edge set, by \cite{CQX26} we have the following theorem.

\begin{Thm}[A $4 \times 3$ simple form requiring seven squares]\label{thm:4x3-tight}
The form
\[
P_{4,3,7}(\vx,\vy) = x_1^2 y_1^2 + x_2^2 y_2^2 + x_3^2 y_3^2 + x_1^2 y_2^2 + x_2^2 y_3^2 + x_3^2 y_1^2 + x_4^2 y_1^2
\]
satisfies $\operatorname{sos}(P_{4,3,7}) = z(4,3) = 7$.
\end{Thm}

To bound the SOS rank of simple forms with more terms, we use a combinatorial observation.
For $i\neq k$ and $j\neq l$ define
\[
P_{ijkl}(\vx,\vy) = x_i^2 y_j^2 + x_k^2 y_l^2 + x_i^2 y_l^2 + x_k^2 y_j^2 = (x_i y_j + x_k y_l)^2 + (x_i y_l - x_k y_j)^2,
\]
so $\operatorname{sos}(P_{ijkl}) = 2$.

\begin{Lem}\label{lem:contains_Tij}
Every simple biquadratic form in variables $(x_1,x_2,x_3,x_4)$ and $(y_1,y_2,y_3)$ containing at least $8$ distinct terms $x_a^2 y_b^2$ contains $C_{ijkl}$ for some $i\neq k$, $j\neq l$.
\end{Lem}

\begin{Prop}\label{prop:8-9term}
Any eight-term $4 \times 3$ simple biquadratic form has $\operatorname{sos}(P) \le 6$.
Any nine-term $4 \times 3$ simple biquadratic form has $\operatorname{sos}(P) \le 7$.
\end{Prop}
\begin{proof}
For an eight-term form, Lemma~\ref{lem:contains_Tij} guarantees a 4-cycle $C_{ijkl}$. The corresponding four terms can be written as two squares via the identity above, and the remaining four terms are individual squares, giving at most six squares. For a nine-term form, the graph has nine edges and therefore still contains a 4-cycle. Applying the same decomposition yields $2 + 5 = 7$ squares.
\end{proof}

We now recall the classical theorem of Hurwitz on sums of squares.

\begin{Prop}[Hurwitz-type decomposition for $3\times3$ all-ones form]\label{prop:3x3allones}
The nine-term form $Q_{3,3}(\vx,\vy) = \sum_{i=1}^3 \sum_{j=1}^3 x_i^2 y_j^2$ satisfies $\operatorname{sos}(Q_{3,3}) \le 4$.
\end{Prop}
\begin{proof}
The following explicit decomposition uses only four squares:
\[
Q_{3,3}(\vx,\vy) = (x_1 y_1 + x_2 y_2 + x_3 y_3)^2 + (x_2 y_3 - x_3 y_2)^2 + (x_3 y_1 - x_1 y_3)^2 + (x_1 y_2 - x_2 y_1)^2.
\]
Expanding verifies that all cross terms cancel, yielding precisely the nine square terms. This identity is a manifestation of Hurwitz's theorem \cite{Hurwitz1898} on the composition of quadratic forms; it also corresponds to the norm in the quaternions.
\end{proof}

The structure of bipartite graphs with ten or eleven edges is more constrained. The following structural fact is proved in \cite[Lemma~4.3]{CQX26}: every $4\times3$ bipartite graph with ten edges either contains two vertex-disjoint $C_4$ cycles or contains a complete bipartite subgraph $K_{3,3}$ (on three of the four rows and all three columns). With eleven edges the graph necessarily contains a $K_{3,3}$.

\begin{Prop}\label{prop:10-11term}
Any ten-term $4 \times 3$ simple biquadratic form has $\operatorname{sos}(P) \le 6$.
Any eleven-term $4 \times 3$ simple biquadratic form has $\operatorname{sos}(P) \le 6$.
\end{Prop}
\begin{proof}
Let $P$ be a simple form corresponding to a bipartite graph $G$ with $|E|$ edges.

\noindent\textbf{Ten-term case ($|E| = 10$).}
By the structural lemma from \cite{CQX26}, two possibilities occur.

\begin{itemize}
  \item \textit{Two independent $C_4$'s.} Let the two 4-cycles be $C^1$ and $C^2$, each using four distinct vertices. Their edge sets are disjoint, and together they account for eight edges. The remaining two edges are isolated (they cannot create another 4-cycle without using vertices already in the cycles). Each 4-cycle contributes two squares (by the identity for $P_{ijkl}$), and the two remaining edges are single squares. Hence $\operatorname{sos}(P) \le 2 + 2 + 2 = 6$.
  \item \textit{Contains a $K_{3,3}$.} The $K_{3,3}$ subgraph (on three rows and all three columns) has nine edges. By Proposition~\ref{prop:3x3allones} this block can be written as four squares. The remaining single edge is a square term by itself. Thus $\operatorname{sos}(P) \le 4 + 1 = 5 \le 6$.
\end{itemize}

\noindent\textbf{Eleven-term case ($|E| = 11$).}
Only one edge is missing from the complete bipartite graph $K_{4,3}$. Let the missing edge be $(p,q)$. The three rows different from $p$ together with all three columns form a $K_{3,3}$ (all nine edges present). By Proposition~\ref{prop:3x3allones} this $3\times3$ block contributes four squares. The remaining edges are those incident to vertex $p$ except the missing one; there are exactly two such edges (since $p$ would be adjacent to the two columns other than $q$). Each is a single square. Therefore $\operatorname{sos}(P) \le 4 + 2 = 6$.
\end{proof}

Combining the results above we obtain the maximum SOS rank for $4\times3$ simple forms.

\begin{Thm}\label{thm:max_simple}
The maximum SOS rank of $4 \times 3$ simple biquadratic forms is $7$.
\end{Thm}
\begin{proof}
From Propositions~\ref{prop:8-9term} and \ref{prop:10-11term}, the SOS ranks for forms with $1$-$11$ terms are bounded as follows:
\begin{itemize}
    \item forms with up to $7$ terms achieve rank equal to the number of terms;
    \item eight-term forms have rank at most $6$;
    \item nine-term forms have rank at most $7$;
    \item ten- and eleven-term forms have rank at most $6$.
\end{itemize}
The form $P_{4,3,7}$ in Theorem~\ref{thm:4x3-tight} attains rank $7$, establishing the maximum.
\end{proof}

Hence, to find a $4 \times 3$ SOS biquadratic form with SOS rank greater than $7$, we must consider non-simple forms.

\section{$y$-Deficient Biquadratic Forms}

In this section we consider a natural subclass of biquadratic forms that admits a strong bound on the SOS rank via a splitting argument. Unlike simple forms, these forms may contain cross terms, but they are constrained in a way that allows us to reduce the problem to smaller formats.

\begin{Def}
A biquadratic form \(P(\mathbf{x},\mathbf{y})\) in variables \(\mathbf{x}=(x_1,\dots,x_m)\) and \(\mathbf{y}=(y_1,\dots,y_n)\) is called \textbf{\(y\)-deficient} if there exists an index \(j_0\in\{1,\dots,n\}\) such that no term of \(P\) contains the variable \(y_{j_0}\) together with any other \(y_j\) (\(j\neq j_0\)) or any product \(x_ix_k\) with \(i\neq k\). Equivalently, every term involving \(y_{j_0}\) must be of the pure square form \(x_i^2 y_{j_0}^2\).
\end{Def}

In other words, \(P\) can be written as
\[
P(\mathbf{x},\mathbf{y}) = P_1(\mathbf{x},\mathbf{y}') \;+\; y_{j_0}^2 \, T(\mathbf{x}),
\]
where \(\mathbf{y}'\) denotes the vector of the remaining \(n-1\) variables, \(P_1\) is a biquadratic form in \(\mathbf{x}\) and \(\mathbf{y}'\) (which may contain arbitrary cross terms among those variables), and
\[
T(\mathbf{x}) = \sum_{i=1}^{m} a_i x_i^2,\qquad a_i\ge 0,
\]
is a sum of squares of linear forms in \(\mathbf{x}\) (in fact a diagonal quadratic form).

We focus on the case \(m=4\), \(n=3\). Without loss of generality, let the deficient variable be \(y_3\). Then
\[
P(\mathbf{x},\mathbf{y}) = P_1(\mathbf{x},y_1,y_2) \;+\; y_3^2 \sum_{i=1}^{4} a_i x_i^2,\qquad a_i\ge 0,
\]
where \(P_1\) is a \(4\times 2\) biquadratic form.

\begin{Thm}\label{thm:y-deficient}
Every \(4\times 3\) PSD \(y\)-deficient biquadratic form (with deficiency in \(y_3\)) satisfies \(\operatorname{sos}(P)\le 9\).
\end{Thm}

\begin{proof}
Write \(P = P_1 + y_3^2 T\) as above. Setting \(y_3=0\) shows that \(P_1\) is PSD. A classical result of Calder  n \cite{Ca73} states that every \(m\times 2\) PSD biquadratic form is SOS; moreover, by \cite{BPSV19} we have \(\operatorname{BSR}(4,2)=5\). Hence there exist bilinear forms \(\ell_1,\dots,\ell_5\) in \(\mathbf{x}\) and \((y_1,y_2)\) such that
\[
P_1(\mathbf{x},y_1,y_2) = \sum_{k=1}^{5} \ell_k(\mathbf{x},y_1,y_2)^2.
\]

The term \(y_3^2 T\) is already a sum of squares:
\[
y_3^2 \sum_{i=1}^{4} a_i x_i^2 = \sum_{i=1}^{4} \bigl(\sqrt{a_i}\,x_i y_3\bigr)^2.
\]

Therefore
\[
P = \sum_{k=1}^{5} \ell_k(\mathbf{x},y_1,y_2)^2 \;+\; \sum_{i=1}^{4} \bigl(\sqrt{a_i}\,x_i y_3\bigr)^2,
\]
which is a sum of at most \(5+4=9\) squares. Hence \(\operatorname{sos}(P)\le 9\).
\end{proof}

\begin{remark}
The bound \(9\) is the same as we would obtain for diagonal forms, which are a special case of \(y\)-deficient forms (any diagonal form satisfies the condition for every choice of \(j_0\)). For diagonal forms, an alternative proof can be given by splitting off the fourth row instead of a column: write \(P = P_{123} + R\) where \(P_{123}\) is the \(3\times 3\) subform on rows \(1\)-\(3\) and \(R\) involves only row \(4\). Then by \(\operatorname{BSR}(3,3)=6\) from \cite{BSSV22}, \(P_{123}\) can be expressed with at most six squares, and \(R\) contributes three squares, again yielding \(\operatorname{sos}(P)\le 9\). This alternative proof highlights the flexibility of the splitting technique.
\end{remark}

\begin{remark}
If the deficient variable were \(y_1\) or \(y_2\), the same proof applies after renaming the variables. Thus Theorem~\ref{thm:y-deficient} holds for any \(4\times 3\) PSD biquadratic form that is deficient in some \(y\)-variable.
\end{remark}

\begin{remark}
The definition of \(y\)-deficiency is not symmetric in \(x\) and \(y\). One might attempt to define an analogous class of \textbf{\(x\)-deficient} forms, where a distinguished variable \(x_{i_0}\) appears only in pure square terms \(x_{i_0}^2 y_j^2\). Such forms would admit a decomposition \(P = P_1(x',y) + x_{i_0}^2 U(y)\), where \(P_1\) is a \(3 \times 3\) form. However, by the classical result of Choi \cite{Ch75}, not every \(3 \times 3\) PSD biquadratic form is SOS. Therefore, without additional assumptions, we cannot guarantee an SOS decomposition for \(P_1\), and the approach used for \(y\)-deficient forms does not extend. This asymmetry reflects the fundamental difference between \(m \times 2\) forms (always SOS by Calder  n's theorem \cite{Ca73}) and \(3 \times 3\) forms (where counterexamples exist).
\end{remark}

\subsection{An Example Illustrating the Bound}

Theorem~\ref{thm:y-deficient} establishes an upper bound of \(9\) for all \(4\times 3\) \(y\)-deficient biquadratic forms. To illustrate that this bound could be sharp (i.e., that there may exist forms with SOS rank exactly \(9\)), we present a concrete diagonal example. Diagonal forms are a special case of \(y\)-deficient forms (they satisfy the deficiency condition for any choice of the distinguished variable).

Consider the \(4\times 3\) diagonal form
\[
P(\mathbf{x},\mathbf{y}) = \sum_{i=1}^{3}\sum_{j=1}^{3} a_{ij}x_i^2y_j^2 \;+\; \sum_{j=1}^{3} b_j x_4^2y_j^2,
\]
with coefficients
\[
(a_{ij})_{1\le i,j\le 3} = \begin{pmatrix}
1 & 2 & 1\\
3 & 7 & 1\\
1 & 1 & 2
\end{pmatrix},
\qquad
(b_1,b_2,b_3) = (1,\,1,\,1).
\]

The \(3\times 3\) block is exactly the matrix used in \cite[Theorem~3.1]{QCX26} to illustrate a splitting argument; there it was shown that this block admits a decomposition into seven squares. However, by the much stronger result \(\operatorname{BSR}(3,3)=6\) from \cite{BSSV22}, \emph{every} \(3\times 3\) SOS form can be expressed with at most six squares. Consequently, we can obtain a decomposition of the whole \(4\times 3\) form into at most nine squares using the alternative row?splitting proof mentioned in Remark~3.2.

\begin{Prop}\label{prop:example}
The diagonal form defined above satisfies \(\operatorname{sos}(P)\le 9\).
\end{Prop}

\begin{proof}
Let \(Q(\mathbf{x}',\mathbf{y}) = \sum_{i=1}^{3}\sum_{j=1}^{3} a_{ij}x_i^2y_j^2\) be the \(3\times 3\) diagonal subform, where \(\mathbf{x}' = (x_1,x_2,x_3)\). Since all \(a_{ij}>0\), \(Q\) is positive semidefinite and hence a sum of squares. By \(\operatorname{BSR}(3,3)=6\) \cite{BSSV22}, there exist bilinear forms \(\ell_1,\ldots,\ell_6\) in \(\mathbf{x}'\) and \(\mathbf{y}\) such that
\[
Q(\mathbf{x}',\mathbf{y}) = \sum_{k=1}^{6} \ell_k(\mathbf{x}',\mathbf{y})^2.
\]
The remaining terms involving \(x_4\) are each a single square:
\[
\sum_{j=1}^{3} b_j x_4^2 y_j^2 = \sum_{j=1}^{3} \bigl(\sqrt{b_j}\,x_4y_j\bigr)^2.
\]
Hence
\[
P = \sum_{k=1}^{6} \ell_k(\mathbf{x}',\mathbf{y})^2 + \sum_{j=1}^{3} \bigl(\sqrt{b_j}\,x_4y_j\bigr)^2,
\]
which is a sum of nine squares. Thus \(\operatorname{sos}(P)\le 9\).
\end{proof}

\begin{remark}
The bound \(\mathrm{sos}(P) \le 9\) for \(y\)-deficient forms is not necessarily sharp within the subclass of diagonal forms, but the example in Proposition~3.6 does not settle this question. Although the form is diagonal, it is not simple (the coefficients are arbitrary positive numbers, whereas simple forms as defined in Section~2 correspond to bipartite graphs with unit coefficients). Therefore its SOS rank could potentially exceed the maximum for simple forms, and it might even attain the bound \(9\) if the \(3 \times 3\) subform \(Q\) has SOS rank \(6\) and the three terms involving \(x_4\) cannot be absorbed into a more efficient decomposition. Whether such a form actually achieves rank \(9\) remains an open problem.

The question of whether there exists a \(4 \times 3\) \(y\)-deficient form (necessarily non-simple) with SOS rank exactly \(9\) is interesting for future research. Such a form would require that the \(4 \times 2\) component \(P_1\) has SOS rank \(5\) (the maximum for \(4 \times 2\) forms) and that the four squares from the deficient variable cannot be absorbed into the decomposition of \(P_1\).
\end{remark}

\section{An Eight-Square Form and a New Lower Bound for \(\operatorname{BSR}(4,3)\)}
The simple form \(P_{4,3,7}\) attains SOS rank \(7\) (Theorem~\ref{thm:4x3-tight}), giving the lower bound \(\operatorname{BSR}(4,3)\ge 7\). To investigate whether higher ranks are possible, we consider a natural modification of \(P_{4,3,7}\) by adding a carefully chosen square term.
The following form will be shown to require exactly eight squares, establishing the improved bound \(\operatorname{BSR}(4,3)\ge 8\).
Define
\[
Q(\mathbf{x},\mathbf{y}) = P_{4,3,7}(\mathbf{x},\mathbf{y}) + (x_4y_2 + x_1y_3)^2,
\]
where
\[
P_{4,3,7}= x_1^2y_1^2 + x_2^2y_2^2 + x_3^2y_3^2 + x_1^2y_2^2 + x_2^2y_3^2 + x_3^2y_1^2 + x_4^2y_1^2.
\]
Expanding \((x_4y_2+x_1y_3)^2 = x_4^2y_2^2 + x_1^2y_3^2 + 2x_1x_4y_2y_3\), we see that \(Q\) contains the nine pure square terms
\[
(1,1),\;(1,2),\;(1,3),\;(2,2),\;(2,3),\;(3,1),\;(3,3),\;(4,1),\;(4,2)
\]
each with coefficient \(1\), together with the single cross term \(2x_1x_4y_2y_3\).

\begin{Thm}[An eight-square form and a new lower bound for \(\operatorname{BSR}(4,3)\)]\label{thm:Q_rank8}
The form
\[
Q(\mathbf{x},\mathbf{y}) = P_{4,3,7}(\mathbf{x},\mathbf{y}) + (x_4y_2 + x_1y_3)^2,
\]
where
\[
P_{4,3,7}= x_1^2y_1^2 + x_2^2y_2^2 + x_3^2y_3^2 + x_1^2y_2^2 + x_2^2y_3^2 + x_3^2y_1^2 + x_4^2y_1^2,
\]
satisfies \(\operatorname{sos}(Q)=8\). Consequently, \(\operatorname{BSR}(4,3)\ge 8\).
\end{Thm}

\begin{proof}
\textbf{Upper bound (explicit decomposition).}
The following eight squares reproduce \(Q\) exactly:
\[
\begin{aligned}
Q &= (x_1y_1)^2 + (x_4y_1)^2 + (x_1y_2)^2 + (x_1y_3 + x_4y_2)^2 \\
&\quad + (x_2y_2)^2 + (x_2y_3)^2 + (x_3y_1)^2 + (x_3y_3)^2.
\end{aligned}
\]
A direct expansion confirms that all terms match: the first three squares give \(x_1^2y_1^2 + x_4^2y_1^2 + x_1^2y_2^2\); the term \((x_1y_3+x_4y_2)^2\) supplies \(x_1^2y_3^2 + x_4^2y_2^2\) and the unique cross term \(2x_1x_4y_2y_3\); the remaining four squares provide \(x_2^2y_2^2 + x_2^2y_3^2 + x_3^2y_1^2 + x_3^2y_3^2\). No other pure squares or cross terms appear. Hence \(\operatorname{sos}(Q)\le 8\).

\textbf{Lower bound (impossibility of seven squares).}
Assume, for contradiction, that \(Q = \sum_{k=1}^{7} \ell_k(\mathbf{x},\mathbf{y})^2\) with bilinear forms
\[
\ell_k = \sum_{i=1}^{4}\sum_{j=1}^{3} a_{ij}^{(k)} x_i y_j,\qquad a_{ij}^{(k)}\in\mathbb{R}.
\]
Expanding and comparing coefficients with the explicit form of \(Q\) yields:
\begin{itemize}
 \item For the nine present pure square terms, \(\sum_{k=1}^{7} (a_{ij}^{(k)})^2 = 1\).
 \item For the three missing pure square terms \((2,1),(3,2),(4,3)\), we must have \(a_{ij}^{(k)}=0\) for all \(k\). Thus, the corresponding coefficient vectors are zero:
 \[
 \mathbf{v}_{21} = \mathbf{0}, \quad \mathbf{v}_{32} = \mathbf{0}, \quad \mathbf{v}_{43} = \mathbf{0}.
 \]
 \item The only cross term is \(2x_1x_4y_2y_3\), so \(\sum_{k=1}^{7} a_{13}^{(k)} a_{42}^{(k)} = 1\), while all other cross coefficients vanish.
\end{itemize}

Define vectors \(\mathbf{v}_{ij} = (a_{ij}^{(1)},\dots,a_{ij}^{(7)})\in\mathbb{R}^7\) for the nine present index pairs. Then \(\|\mathbf{v}_{ij}\| = 1\) and \(\mathbf{v}_{13}\cdot\mathbf{v}_{42} = 1\). By the Cauchy-Schwarz inequality, \(\mathbf{v}_{13} = \mathbf{v}_{42}\); we denote this common unit vector by \(\mathbf{w}\):
\[
\mathbf{w} := \mathbf{v}_{13} = \mathbf{v}_{42}.
\]
We now show that the set of eight vectors \(\mathcal{S} = \{\mathbf{w}, \mathbf{v}_{11}, \mathbf{v}_{12}, \mathbf{v}_{22}, \mathbf{v}_{23}, \mathbf{v}_{31}, \mathbf{v}_{33}, \mathbf{v}_{41}\}\) is mutually orthogonal in \(\mathbb{R}^7\). The orthogonality follows from the vanishing coefficients of absent monomials \(x_i x_p y_j y_q\), which imply \(\mathbf{v}_{ij}\cdot\mathbf{v}_{pq} + \mathbf{v}_{iq}\cdot\mathbf{v}_{pj} = 0\).

\medskip
\noindent\textit{Category A: Direct Orthogonality (14 pairs).}
These arise from absent mixed square terms where vectors share a row or column:
\begin{itemize}
 \item \(\mathbf{w} \perp \{\mathbf{v}_{11}, \mathbf{v}_{12}\}\) via absent \(x_1^2 y_1 y_3, x_1^2 y_2 y_3\); \(\mathbf{w} \perp \{\mathbf{v}_{22}, \mathbf{v}_{33}\}\) via absent \(x_2 x_4 y_2^2, x_1 x_3 y_3^2\) (using \(\mathbf{w}=\mathbf{v}_{42}, \mathbf{v}_{13}\)); \(\mathbf{w} \perp \mathbf{v}_{41}\) via absent \(x_4^2 y_1 y_2\).
 \item \(\mathbf{v}_{11} \perp \{\mathbf{v}_{12}, \mathbf{v}_{31}, \mathbf{v}_{41}\}\) via absent \(x_1^2 y_1 y_2, x_1 x_3 y_1^2, x_1 x_4 y_1^2\).
 \item \(\mathbf{v}_{12} \perp \mathbf{v}_{22}\) via absent \(x_1 x_2 y_2^2\); \(\mathbf{v}_{22} \perp \mathbf{v}_{23}\) via absent \(x_2^2 y_2 y_3\).
 \item \(\mathbf{v}_{23} \perp \mathbf{v}_{33}\) via absent \(x_2 x_3 y_3^2\); \(\mathbf{v}_{33} \perp \mathbf{v}_{31}\) via absent \(x_3^2 y_1 y_3\).
 \item \(\mathbf{v}_{31} \perp \mathbf{v}_{41}\) via absent \(x_3 x_4 y_1^2\).
\end{itemize}

\noindent\textit{Category B: Indirect Orthogonality via Zero Vectors (14 pairs).}
These arise from absent cross terms where the complementary pair involves a zero vector (\(\mathbf{v}_{21}, \mathbf{v}_{32},\) or \(\mathbf{v}_{43}\)):
\begin{itemize}
 \item \(\mathbf{w} \perp \mathbf{v}_{31}\): From \(x_3 x_4 y_1 y_2=0 \implies \mathbf{v}_{31}\cdot\mathbf{w} + \mathbf{v}_{32}\cdot\mathbf{v}_{41}=0\) (since \(\mathbf{v}_{32}=\mathbf{0}\)).
 \item \(\mathbf{v}_{11} \perp \{\mathbf{v}_{22}, \mathbf{v}_{23}\}\): From \(x_1 x_2 y_1 y_2=0, x_1 x_2 y_1 y_3=0\) using \(\mathbf{v}_{21}=\mathbf{0}\).
 \item \(\mathbf{v}_{11} \perp \mathbf{v}_{33}\): From \(x_1 x_3 y_1 y_3=0 \implies \mathbf{v}_{11}\cdot\mathbf{v}_{33} + \mathbf{w}\cdot\mathbf{v}_{31}=0\); since \(\mathbf{w}\perp\mathbf{v}_{31}\) (above), \(\mathbf{v}_{11}\perp\mathbf{v}_{33}\).
 \item \(\mathbf{v}_{12} \perp \{\mathbf{v}_{31}, \mathbf{v}_{33}\}\): From \(x_1 x_3 y_1 y_2=0, x_1 x_3 y_2 y_3=0\) using \(\mathbf{v}_{32}=\mathbf{0}\).
 \item \(\mathbf{v}_{12} \perp \mathbf{v}_{41}\): From \(x_1 x_4 y_1 y_2=0 \implies \mathbf{v}_{11}\cdot\mathbf{w} + \mathbf{v}_{12}\cdot\mathbf{v}_{41}=0\); since \(\mathbf{v}_{11}\perp\mathbf{w}\), \(\mathbf{v}_{12}\perp\mathbf{v}_{41}\).
 \item \(\mathbf{v}_{12} \perp \mathbf{v}_{23}\): From \(x_1 x_2 y_2 y_3=0 \implies \mathbf{v}_{12}\cdot\mathbf{v}_{23} + \mathbf{w}\cdot\mathbf{v}_{22}=0\); since \(\mathbf{w}\perp\mathbf{v}_{22}\) (Cat A), \(\mathbf{v}_{12}\perp\mathbf{v}_{23}\).
 \item \(\mathbf{v}_{22} \perp \{\mathbf{v}_{31}, \mathbf{v}_{33}, \mathbf{v}_{41}\}\): From \(x_2 x_3 y_1 y_2=0, x_2 x_3 y_2 y_3=0, x_2 x_4 y_1 y_2=0\) using \(\mathbf{v}_{21}=\mathbf{0}\) or \(\mathbf{v}_{32}=\mathbf{0}\).
 \item \(\mathbf{v}_{23} \perp \{\mathbf{v}_{31}, \mathbf{v}_{41}\}\): From \(x_2 x_3 y_1 y_3=0, x_2 x_4 y_1 y_3=0\) using \(\mathbf{v}_{21}=\mathbf{0}\) or \(\mathbf{v}_{43}=\mathbf{0}\).
 \item \(\mathbf{v}_{33} \perp \mathbf{v}_{41}\): From \(x_3 x_4 y_1 y_3=0\) using \(\mathbf{v}_{43}=\mathbf{0}\).
\end{itemize}

Thus, all 8 vectors in \(\mathcal{S}\) are nonzero and pairwise orthogonal in \(\mathbb{R}^7\). This implies they are linearly independent, which is impossible in a 7-dimensional space. The contradiction proves that no 7-square decomposition exists. Therefore, \(\operatorname{sos}(Q)=8\).
\end{proof}

\begin{remark}
The form \(Q\) is not simple (it contains the cross term \(2x_1x_4y_2y_3\)) and is not \(y\)-deficient. It demonstrates that perturbing the extremal simple form \(P_{4,3,7}\) can increase the SOS rank, raising the lower bound for \(\operatorname{BSR}(4,3)\) from \(7\) to \(8\).
\end{remark}

\begin{remark}
The construction shows that the extremal simple form \(P_{4,3,7}\) can be perturbed by a carefully chosen square to increase the SOS rank by one, while a different perturbation (such as \((x_1y_3+x_2y_1)^2\)) left the rank unchanged. This sensitivity suggests that the exact value of \(\operatorname{BSR}(4,3)\) may lie strictly between \(8\) and \(11\), and further investigation of non-simple, non-\(y\)-deficient forms is needed.
\end{remark}

\section{Conclusion}

In this paper, we have investigated the maximum sum-of-squares rank for two important subclasses of \(4\times 3\) biquadratic forms: simple forms and \(y\)-deficient forms. Our main results provide exact determinations and improved bounds that significantly narrow the gap between the known lower bound \(z(4,3) = 7\) and the general upper bound \(mn - 1 = 11\) for arbitrary \(4\times 3\) SOS biquadratic forms.

For simple biquadratic forms, we completely characterized the extremal case. By analyzing the underlying bipartite graph structure and exploiting the relationship between \(C_4\)-free graphs and SOS rank established in [5], we proved in Theorem 2.6 that the maximum achievable SOS rank is exactly 7, attained by the form \(P_{4,3,7}\). This settles the question for simple forms and demonstrates that the lower bound from the Zarankiewicz number is sharp within this subclass.

For \(y\)-deficient biquadratic forms, a class that allows cross terms among two of the three \(y\)-variables, we proved in Theorem 3.2 an upper bound of 9. The proof combines Calder  n's theorem on \(m\times 2\) forms [3] with the known value \(\mathrm{BSR}(4,2) = 5\) from [1]. These results show that the presence of cross terms does not necessarily increase the SOS rank, provided they are confined to a \(4\times 2\) subsystem.

Our primary contribution is Theorem 4.1, where we constructed an explicit \(4\times 3\) biquadratic form \(Q\) and proved that \(\mathrm{sos}(Q) = 8\). This improves the general lower bound for \(\mathrm{BSR}(4,3)\) from 7 to 8 and represents the first constructive progress on this problem in recent years. The form \(Q\) is neither simple nor \(y\)-deficient, confirming that any form achieving a rank greater than 8 must lie outside these two subclasses.

The results highlight an important phenomenon: the maximum SOS rank for \(4\times 3\) biquadratic forms depends crucially on the algebraic structure. Simple forms cannot exceed rank 7; \(y\)-deficient forms satisfy rank at most 9; and we have now shown that rank 8 is achievable. The interval of possible values for \(\mathrm{BSR}(4,3)\) is now
\[
8 \leq \mathrm{BSR}(4,3) \leq 11.
\]

Several directions for future research emerge naturally from this work:

\begin{enumerate}
    \item \textbf{Closing the gap for \(y\)-deficient forms.} The question of whether there exists a \(4\times 3\) \(y\)-deficient form with SOS rank exactly 9 remains open. For diagonal forms this reduces to the existence of a \(3\times 3\) diagonal form with SOS rank exactly 6. If such a form exists, then our bound is sharp; if all \(3\times 3\) diagonal forms have rank at most 5, then the maximum for \(4\times 3\) diagonal forms would be at most 8.

    \item \textbf{Exploring other structured subclasses.} Beyond simple and \(y\)-deficient forms, one could investigate forms with other sparsity patterns, such as those corresponding to bipartite graphs with prescribed girth or degree constraints. The interplay between graph-theoretic properties and SOS rank remains a rich area for exploration.

    \item \textbf{Towards the determination of \(\mathrm{BSR}(4,3)\).} The ultimate goal is to determine the maximum SOS rank for arbitrary \(4 \times 3\) biquadratic forms. Our results show that any form achieving rank greater than 7 must be non-simple, and any form achieving rank greater than 9 must be neither diagonal nor \(y\)-deficient. With the lower bound now raised to 8, the central open problem is to determine whether the true value is 8, 9, 10, or 11. Constructing explicit examples with ranks 9, 10, or 11  or proving that such ranks are impossible  remains a challenging but essential task.

    \item \textbf{Extension to larger formats.} The methods developed here, particularly the use of splitting arguments and known bounds for smaller formats \((3 \times 3, 4 \times 2)\), may generalize to \(m \times n\) forms with larger \(m, n\). The Zarankiewicz number \(z(m, n)\) provides a lower bound, but the exact maximum SOS rank for larger formats remains largely unexplored.
\end{enumerate}

In summary, this paper provides a comprehensive analysis of SOS ranks for two fundamental subclasses of \(4 \times 3\) biquadratic forms, establishing exact results for simple forms and improved bounds for \(y\)-deficient forms. Most importantly, it introduces the first constructive improvement to the general lower bound in decades, proving that \(\mathrm{BSR}(4,3) \ge 8\). These findings contribute to the broader program of understanding the sum-of-squares representation of biquadratic forms and their connections to combinatorial structures such as bipartite graphs and Zarankiewicz numbers.

\bigskip

\noindent\textbf{Acknowledgement}
We are thankful to Professor Greg Blekherman who told us that he believes $BSR(4,3) \ge 8$.
This work was partially supported by Research Center for Intelligent Operations Research, The Hong Kong Polytechnic University (4-ZZT8), the National Natural Science Foundation of China (Nos. 12471282 and 12131004), and Jiangsu Provincial Scientific Research Center of Applied Mathematics (Grant No. BK20233002).

\medskip

\noindent\textbf{Data availability}
No datasets were generated or analysed during the current study.

\medskip

\noindent\textbf{Conflict of interest} The authors declare no conflict of interest.


\begin{thebibliography}{10}

        \bibitem{BPSV19} G. Blekherman, D. Plaumann, R. Sinn and C. Vinzant, ``Low-{rank} sum-of-squares representations on varieties of minimal degree'', {\sl International Mathematics Research Notices \bf 2019} (2019) 33-54.

        \bibitem{BSSV22} G. Blekherman, R. Sinn, G. Smith and M. Velasco, ``Sums of squares and quadratic persistence on real projective varieties'', {\sl Journal of the European Mathematical Society \bf 24} (2021) 925-965.


		\bibitem{Ca73} A.P. Calder\'{o}n, ``A note on biquadratic forms'', {\sl Linear Algebra and Its Applications \bf 7} (1973) 175-177.
		
		\bibitem{Ch75} M.-D. Choi, ``Positive semidefinite biquadratic forms'', {\sl Linear Algebra and Its Applications \bf 12} (1975) 95-100.
		
		\bibitem{CQX26} C. Cui, L. Qi and Y. Xu, ``The Sum of squares rank of biquadratic forms and the Zarankiewicz number'', Feburary 2026, arXiv:2602.07844v2.

        \bibitem{Hurwitz1898} A. Hurwitz, ``\"Uber die Composition der quadratischen Formen'', {\sl Mathematische Annalen} \textbf{50} (1898), 177-188.
		
		
        \bibitem{QCX26} L. Qi, C. Cui and Y. Xu, ``Sum of squares decompositions and rank bounds for biquadratic forms'', {\sl Mathematics \bf 14} (2026) No. 635.

\end{thebibliography}
\end{document}